\documentclass[12pt,leqno]{amsart}
\usepackage[dvips]{graphics}
\usepackage{amssymb}
\usepackage{verbatim}
\usepackage{hyperref}
\usepackage{color}

\numberwithin{equation}{section}
\newdimen\vintkern\vintkern12pt
\def\vint{-\kern-\vintkern\int}

\newtheorem{thm}{Theorem}[section]

\newtheorem{lem}[thm]{Lemma}

\newtheorem{cor}[thm]{Corollary}
\newtheorem{prop}[thm]{Proposition}

\newcommand{\tref}[1]{Theorem~\ref{#1}}
\newcommand{\cref}[1]{Corollary~\ref{#1}}

\newcommand{\R}{\mathbb{R}}

\newcommand{\lip}{\mathrm{Lip}}

\begin{document}

	\pagebreak

	\title{Some regularity of Submetries}
	
\thanks{
	A. L. was partially supported by the DFG grants   SFB TRR 191 and SPP 2026.}

	\author{Alexander Lytchak}
	
	\address
{Institute of Algebra and Geometry\\ KIT\\ Englerstr. 2\\ 76131 Karlsruhe, Germany}
	\email{alexander.lytchak@kit.edu}


	\begin{abstract}
We discuss   regularity statements for  equidistant decompositions of 
 Riemannian manifolds and for the corresponding quotient spaces. 
We show that any stratum of the quotient space
has  curvature locally bounded from both sides. 
	\end{abstract}
	
	\maketitle

	\renewcommand{\theequation}{\arabic{section}.\arabic{equation}}
	\pagenumbering{arabic}

	\section{Introduction}
	 A  \emph{submetry} is a map $P:X\to Y$  between metric spaces  which 
sends open balls $B_r(x)$  in $X$ onto  balls $B_r (P(x))$ of the same radius  in $Y$.
Submetries as  metric generalization of Riemannian submersions have been introduced by  Berestovskii \cite{Berest2}.   Berestovksii and  Guijarro  verified that a submetry between smooth complete Riemannian manifolds always is  a $\mathcal C^{1,1}$ Riemannian submersion, but  it does not need to be $\mathcal C^2$
 \cite{Berest}.  Another example of a submetry is provided  by a distance function 
$P:\R^n\to \R$ to a convex nowhere dense subset $C\subset \R^n$.

Submetries $P:X\to Y$ with given \emph{total space} $X$ are in one-to-one correspondence with equidistant decompositions
of  $X$. The correspondence assigns to $P$ the decomposition of $X$ into fibers of $P$ \cite[Section 2.2]{KL}.    Seen this way, submetries generalize quotient maps for isometric group actions and decompositions of a complete smooth Riemannian manifold into leaves of a singular Riemannian foliation with closed leaves \cite{Molino}, \cite{Thorb},  \cite{Alexandrino-survey}.

Recent appearances
of submetries in many unrelated settings: \cite{Lyt-buildings}, \cite{LPZ}, \cite{Stadler}, \cite{StadlerII} 
\cite{Ber-finite},
 \cite{Mendes-Rad}, \cite{Mendes-Rad2}, \cite{Grovesubmet},
  \cite{GW}, \cite{Rade}, 
	 \cite{GW}, 
 make investigations of the properties of submetries  a  natural task, especially if the total space is a Riemannian manifold.
A systematic study of submetries $P:M\to Y$ with total space a sufficiently smooth Riemannian manifold has been intitiated in \cite{KL}. The present paper continues the investigations of \cite{KL} and improves some regularity statements provided there.

If the total space  $X$ is a connected (sufficiently smooth) complete Riemannian manifold $M$,  the following structural results on the base space $Y$ of a submetry $P:M\to Y$  have been derived in \cite{KL}.

The quotient space $Y$ locally has curvature bounded from 
below,   \cite{BGP}, \cite[Proposition 3.1]{KL}. There is a canonical stratification 
$Y= \cup _{l=0} ^m Y^l$, where $m$ is the dimension of $Y$
and $Y^l$ consists of all points $y\in Y$ such that the tangent space $T_yY$ has $\R^l$ as a direct factor, \cite[Theorem 1.6]{KL}. The subset $Y^l$ is locally convex in $Y$, for any $l$, and  it is an $l$-dimensional manifold.  The maximal-dimensional stratum $Y^m$, the set of
\emph{regular points} of $Y$, is open, dense and convex in $Y$.

For any point $y\in Y$ there exists some  $r>0$, such that exponential map $\exp_y$ is a  well-defined homeomorphism $\exp _y:B_r(0) \to B_r(y)$
between the $r$-ball in the tangent cone $T_yY$ around the origin and the $r$-ball 
in $Y$ around $y$, \cite[Theorem 1.3]{KL}.
This  \emph{injectivity radius} is locally bounded from below on each stratum $Y^l$,
but it goes  to $0$, when points on $Y^l$ converge to a lower-dimensional stratum.

Our first result improves the regularity of the exponential map:

\begin{thm} \label{thm: exp}
Let $Y$ be a base of a submetry $P:M\to Y$ of a Riemannian manifold with locally bounded curvature.
Then, for any $y\in Y$, there exist $r_0,C>0$, such that for all $r<r_0$ the exponential map $\exp _y:B_r(0)\to B_r(y)$ is $(1+Cr^2)$-bilipschitz. 
\end{thm}

Here and below, we use the notion of   a \emph{Riemannian manifold with locally bounded curvature} to describe a  manifold without boundary with a continuous Riemannian metric,
which has  curvature bounded locally from above and below  in the sense of Alexandrov see \cite{Ber-Nik}, \cite{KL-both}. Any $\mathcal C^{1,1}$-submanifold of a Riemannian manifold with a $\mathcal C^{1,1}$-Riemannian metric is in this class \cite[Proposition 1.7]{KL-both}.

 Theorem \ref{thm: exp} can be informally   understood as the existence of
 a pointwise  both-sided curvature bound  at any point $y\in Y$. Indeed, for a smooth Riemannian manifold $M=Y$, the optimal number $C$ in the statement of Theorem \ref{thm: exp} is  equivalent  (up to a factor) to  the optimal bound on the norm of the sectional curvatures at $y$.

In Theorem \ref{thm: exp}, the constant  $r_0(y)$ \emph{always} goes to $0$ and $C(y)$ \emph{usually}  goes to infinity, when $y$ converges to a lower stratum, \cite[Proposition 8.9]{KL},  \cite[Theorem 1.1]{Thorb}.
But both constants can be chosen
\emph{locally uniformly}  on any stratum, Theorem \ref{prop: bil} below.  This has the following consequence, which answers \cite[Question 1.12]{KL}:

\begin{cor} \label{cor: expstr}
Let $M$ be a Riemannian manifold with  locally bounded curvature
 and let 
$P:M\to Y$ be a submetry.  Then, any stratum $Y^l$  of $Y$ is a Riemannian manifold with locally bounded   curvature.
\end{cor}

 For smooth Riemannian manifolds $M$ the result will be strengthened in the continuation \cite{LWil}.  If $M$ is analytic,  
the analyticity of the maximal stratum  $Y^m$ has been verified in \cite{Lem}.

In general, fibers of a submetry $P:M\to Y$ can be arbitrary subsets of positive reach in $M$ (this is a common generalization of convex subsets and 
$\mathcal C^{1,1}$ submanifolds \cite{Federer}, \cite{Ly-conv}, \cite{Rataj}).
However, most fibers are $\mathcal C^{1,1}$-submanifolds and any fiber  $L$
of $P$ contains a $\mathcal C^{1,1}$-submanifold, open and dense in $L$. 
A by-product of the proof of Theorem \ref{thm: exp} is the following result saying that for any submetry $P:M\to Y$   the projections from nearby $P$-fibers onto  any manifold $P$-fiber is \emph{almost} a submetry.  We formulate it as a global result for compact fibers and refer to 
Theorem \ref{thm: semicont} for a more general local version.

\begin{prop} \label{prop: almsubm}
Let $P:M\to Y$ be a submetry, where $M$ has locally bounded curvature.
Let $L$ be a fiber of $P$ which is a compact manifold.  Then there exist constants $C,r_0>0$
such that for all fibers  $L'$ of $P$ at distance $r<r_0$ from $L$, the closest point projection  $\Pi^L :L'\to L$ is $(1+Cr)$-Lipschitz and locally $(1+Cr)$-open.
\end{prop}

Recall that a map $f:X\to Y$ between metric spaces is \emph{ locally  $C$-open},  (other terms used are \emph{Lipschitz open} and \emph{co-lipschitz}) if  for any $z\in X$ there exists $r_0>0$, such that,
 for any $r<r_0$ and any $x\in B_{r_0} (z)$,
$$B_r(f(x))\subset f(B_{Cr} (x))\,.$$

A submetry $P:M\to Y$ is called \emph{transnormal} if all fibers of $P$ are $\mathcal C^{1,1}$-submanifolds. Thus, for transnormal submetries with compact fibers  the conclusion of Proposition \ref{prop: almsubm} is true for all fibers.  Moreover, for transnormal submetries, the constants $C,r_0$  appearing in Proposition \ref{prop: almsubm} and in Theorem \ref{thm: exp} depend 
only on the following data:  A bound on the curvature and the injectivity radius of $M$, a bound on the injectivity radius of $Y$ at $y=P(L)$ and a lower volume bound of $Y$ around $y$, see  Corollary \ref{cor: uniformity} below. This seems to be useful for applications
to the theory of Laplacian algebras  developped by Ricardo Mendes and Marco Radeschi,
\cite{MR-Weyl}.

Theorem \ref{thm: exp} implies that the local decomposition of the  base space $Y$ in strata around a point $y$ corresponds to the decomposition in strata of the 
tangent space at $y$, see Corollary \ref{lem: locinf} below. This has the following 
consequence for \emph{transnormal submetries}:

\begin{cor} \label{cor: last}
Let $P:M\to Y$ be a transnormal submetry, where $M$ has locally bounded curvature.
   Let $\gamma:I\to M$ be a horizontal geodesic.  Then, up to discretely many values $t_i\in I$, the connected component of the fiber of $P$ through $\gamma (t)$ has the same dimension $k=k(\gamma)$ and   $P(\gamma (t))$ is contained in the  stratum $Y^l$, with $l=l(\gamma)$.
\end{cor}

  As a related consequence of Proposition \ref{prop: almsubm}, we prove that all holonomy maps between fibers of transnormal submetries are  Lipschitz-open, see Proposition \ref{prop: hol} below.

We mention,  that all results stated here and  below  do not require  completeness of $M$ and are valid for \emph{local submetries}, see   Section \ref{subsec: loc}.

{\bf Acknoweldgements} The author is grateful to Ricardo Mendes and Marco Radeschi  for their interest and helpful comments.

\section{Preliminaries: Manifolds with  bounded curvature}
\subsection{Notation}
By $d$ we denote the distance in metric spaces.
For a subset $A$ of a metric space $X$ we denote by $d_A:X\to \R$ the distance function to  $A$.
A \emph{geodesic} will denote an isometric  (i.e. globally  distance preserving) embedding of an interval.
A local geodesic $\gamma:I\to X$ is a curve whose restrictions to small sub-intervals are geodesics.

\subsection{Curvature bounds and bounds of geometry}
We assume some familiarity with spaces with curvature bounded in the sense of Alexandrov. We refer the reader to \cite{Ber-Nik},  \cite{AKP}.

 By a \emph{manifold with locally bounded curvature}  $M$ we mean a length metric space homeomorphic to a manifold without boundary, such that any point $x\in M$ has a convex neighborhood in $M$,  which is a CAT$(\kappa)$ space and an Alexandrov space of curvature bounded  from below by $-\kappa$, for some $\kappa \in \R$.  We allow the manifold $M$ to be non-complete and the value $\kappa$ to be not globally bounded on $M$.

The \emph{distance coordinates} define a $\mathcal C^{1,1}$-atlas on  any such manifold $M$ and the Riemannian metric is Lipschitz continuous in these coordinates, \cite{Ber-Nik}.
Any $\mathcal C^{1,1}$-submanifold $N\subset M$ also has locally bounded curvature in its intrinsic metric \cite[Proposition 1.7]{KL-both}. 

In any manifold $M$ with locally bounded curvature there is a notion of parallel translation along  any Lipschitz curve \cite[Section 13]{Ber-Nik}.

Let $x$ be a point  in a manifold $M$ with locally bounded curvature and let $\rho >0$ be given. We say that the \emph{the geometry is bounded by $\rho$ at $x$} if the following conditions hold true:

The ball $B=\bar B_{\frac {10} { \rho} } (x)$ is compact, convex and uniquely geodesic and the curvature   in $B$ is bounded from below and above by $\pm \frac  { \rho^2} {100}$.

If the geometry of $M$ at $x$ is bounded by $\rho$, then the metric space $\lambda \cdot M$  rescaled by $\lambda>0$ has at $x$ geometry bounded by $\frac  \rho {\lambda } $.

 Let  the geometry of $M$ at $x$ be bounded by $\rho$ and consider the ball $B=B_{\frac 1 \rho}  (x)$.
Consider some distance coordinates on $B$ and the Lipschitz continuous  metric tensor $g$
defining the metric of $B$ in these coordinates.
Then  there exist a sequence of smooth, uniquely geodesic  metrics $g_n$ on $B$ such that the sectional  curvatures of $g_n$ are bounded in norm by $\frac {\rho^2 } {100} $, with the following properties
 \cite[Section 15]{Ber-Nik}: The metric tensors $g_n$ converge to $g$ in $\mathcal C^{0,1}$
and the parallel transport for $g_n$  uniformly converges to the parallel transport  for $g$.

This approximation result allows us to prove  metric statements in the smooth case first and then to obtain the general case by a limiting procedure. Mostly, a more direct but technical explanation is available without using the approximation theorem.  
The main additional tool available in the smooth situation are Jacobi-fields, which  only have almost everywhere analogues in the general case.
Readers not aquainted 
with the theory of \cite{Ber-Nik} may always assume the  total manifold $M$ to be smooth.
We prefer to stick to the more general setting of manifolds with locally bounded curvature,  since this setting seems to be  appropriate for the study of submetries,
see \cite{KL}.

\subsection{Comparison of tangent vectors at different points}
Let $M$ be a Riemannian manifold  with locally bounded curvature. Let  $O\subset M$ be open, uniquely geodesic and  convex.  Given $x,z\in O$    and vectors $v\in T_xO, w\in T_zO$,  we define  $|v-w|$ to be the distance in $T_xO$ between $v$ and the parallel transport $w'$ of $w$ to $T_xO$ along the geodesic $xz$.

This "quasi-distance" is symmetric but satifies the triangle inequality only up to a defect depending on the geometry  of $O$, see \eqref{eq: tri} below.

For  linear subspaces $W_x\subset T_xO$ and $W_z\subset T_zO$ with $\dim (W_x)=\dim (W_z)$, we  denote by $|W_x-W_z|$ the  symmetric "quasi-distance":
$$|W_z-W_x|:= \sup \{d(W_x,w')  \}\,.$$
Here, the distance $d(W_x,w')$ to the subspace $W_x$ is measured in $T_xO$, and the supremum
is taken over all parallel translates $w'\in W_x$ of unit vectors  $w\in W_z$ along the geodesic $xz$.

\subsection{Almost flat domains} 
We fix $\varepsilon = 10^{-4}$ for the rest of the paper.

We say that $M$ is \emph{almost flat} at $x\in M$ if the geometry of $M$  at $x$ is bounded  by $\varepsilon$.

If $M$ has geometry bounded by $\rho$ at $x$ then, for any $\lambda \geq \frac \rho {\varepsilon }$, the rescaled manifold $\lambda  \cdot M$ is almost flat at $x$.

Let $M$ be almost flat as $x$ and  consider the open ball $O=B_{10} (x)$.  Thus, $O$
is convex and uniquely geodesic and the curvature in $O$ is bounded from both sides by $\pm  \frac {\varepsilon ^2} {100}= \pm 10 ^{-10}$.

We refer to \cite[Section 6]{BK} for the estimates  stated below.

For any ball $B_r(x) \subset O$, the exponential  map $\exp _x :B_r(0)\to B_r(x)$ is
 $(1+{\varepsilon} ^2\cdot  r^2)$-bilipschitz on  $B_r(0) \subset T_xO$, \cite[Proposition 6.4]{BK}.

For a triangle $\Delta=xyz \subset O$ with two sides of length $l_1,l_2$,
the holonomy along $\Delta$  of any  $v \in T_xO$ satisfies \cite[Section 6.2.1]{BK}
\begin{equation} \label{eq: holonomy}
||Hol ^{\Delta} (v)-v||  \leq     \varepsilon ^2 \cdot l_1\cdot l_2 \cdot ||v||\;.
\end{equation}
For such $\Delta$ and arbitrary 
$v_x\in T_xO, v_y\in T_yO, v_z \in T_zO$, set $a:=\min \{||v_x||,||v_y||,||v_z||\}$. 
Then  the holonomy bound  \eqref{eq: holonomy} implies a  triangle inequality with a defect:
\begin{equation} \label{eq: tri}
|v_x-v_z|\leq |v_x-v_y |+|v_y-v_z| +  \varepsilon  ^2\cdot l_1\cdot l_2\cdot a\;.
\end{equation}

The next result is a direct consequence of  \cite[Proposition 6.6]{BK}.
\begin{lem} \label{lem: section2} 
Assume  $B_4(x)\subset O$ and $u,w\in T_xO$ with $||u||, ||w|| < 2$.
 Set $z=\exp _x (w)$ and
$p=\exp _x(u)$ and $q=\exp _x (w+u)$.  Let $\tilde u$ be the parallel translate of
$u$ to $z$ and $q':= \exp _z (\tilde u)$.
Then 
\begin{equation} \label{eq: buserkarcher1}
d(q, q')  \leq  \varepsilon  ^2\cdot (||u||+||w||)\cdot ||u|| \cdot ||w||\;.
\end{equation}
Therefore,
\begin{equation} \label{eq: verylast}
\angle qpq' \leq 2 \cdot  \varepsilon ^2 \cdot ||u|| \cdot (||u||+ ||w||).
\end{equation}
And 
\begin{equation}
|w-\exp _p ^{-1} (q')|\leq 2 \cdot  \varepsilon ^2  \cdot (||u||+ ||w||) \cdot ||u||\cdot ||w||.
\end{equation}
\end{lem}

The estimate  \eqref{eq: buserkarcher1} implies:

\begin{cor} \label{cor: buserkarcher}
Let $w,u\in T_xO$ be given with $||w||= 1$.  Consider the curve $\eta (t):=\exp _x (u+tw)$.
Then, for all sufficiently small $t$, the  starting direction $w_j$ of the geodesic connecting
$\eta (0)$ and $\eta (t)$ satisfies
 $$|w_j-w|  \leq  2\cdot \varepsilon ^2 \cdot ||u||^2 \;.$$
\end{cor}

We will need the following  (definitely not optimal) lemma:

\begin{lem} \label{lem: karcher}
Let $\eta, \gamma:[0,1]\to O$ be geodesics. Set 
 $x_t:=\eta (t)$, $z_t:=\gamma (t)$ and
$a_t=d(x_t,z_t)$. Finally, set
$v_t:=\exp _{x_t} ^{-1} (z_t) \in T_{x_t} M$.

 If  $ 2   > 2a_0 >a_1$  then, for any 
$  1  \geq t >0$,
\begin{equation} \label{eq: lem}
|v_t-v_0| \leq 5 \cdot a_0 \cdot t\;.
\end{equation}
\end{lem}

\begin{proof}
 Let $h\in T_{x_0} O$  be the starting directions of $\gamma $.
 Let $\tilde h$ be the parallel translation of $h$ to $z_0$. 
Set $\tilde z_t=\exp _{z_0} (t\cdot \tilde h)$ and $\tilde  v_t= \exp _{x_t} ^{-1} (\tilde z_t)$.
From \eqref{eq: verylast} we deduce 
$$d(x_1,\tilde z_1)\leq a_0+  2\varepsilon ^2 (1+a_0)\cdot a_0 <2a_0$$
Thus, $d(z_1,\tilde z_1) < 4a_0$. The bilipschitz property of $\exp _{z_0}$ implies
\begin{equation} \label{eq: zt}
d(\tilde z_t, z_t)\leq (4+\varepsilon)\cdot a_0\cdot t\;.
\end{equation}

Applying \eqref{eq: verylast} again, we deduce
$$|\tilde v_t-v_0|\leq 2\varepsilon ^2\cdot (a_0+t)\cdot t\cdot a_0 < 4\varepsilon ^2 \cdot a_0 \cdot t\,.$$
Together with   \eqref{eq: zt} this implies 
$|v_t-v_0|<5a_0\cdot t$.
\end{proof}

\subsection{Subsets of positive reach}
Let $M$ be a Riemannian manifold with locally bounded curvature. 
A locally closed subset $L\subset M$ has
\emph{positive reach} in $M$ if the closest-point projection $\Pi^L$ in uniquely defined on a neighborhood $U$ of $L$ in $M$.   In this case,  $\Pi^L$ is locally Lipschitz on $U$ and the distance function $d_L$ is $\mathcal C^{1,1}$ on $U\setminus L$ \cite{KL-both}.

  A subset $L$  of positive reach  in $M$ is a topological manifold if and only if $L$ is a $\mathcal C^{1,1}$ submanifold of $M$, \cite[Proposition 1.4]{Ly-conv}.  On the other hand, any set $L$ of positive reach contains a subset $L'$ open and dense  in $L$, which is a $\mathcal C^{1,1}$-submanifold, possibly with components of different dimensions \cite[Theorem 7.5]{Rataj}.

The following   result is essentially contained in  \cite[Theorem 1.6, Theorem 1.2]{Ly-conv}.
It formalizes the  following observation: For a $\mathcal C^{1,1}$ submanifold   a lower  bound on 
the reach is equivalent to an upper bound of the second fundamental form, as well as to
a $\mathcal C^{1,1}$-bound of the submanifold. The proof consists just  of a few citations.

\begin{lem}  \label{lem: radeschi}
There exists  $c>0$ with the following properties.

Let the geometry of $M$ at $p$ be bounded by $\rho$. Let $L$ be a closed subset containing $p$, such that the closest point projection $\Pi=\Pi^L$  onto $L$ is uniquely defined in $B_{\frac {10} {\rho}} (p)$.

Then, 
for  $r\leq \frac 3 \rho$, 
 the  Lipschitz constant of $\Pi$ on $B_r(p)$ is at most
\begin{equation} \label{eq: lip}
\lip (\Pi^{L}  :B_r (p)\to L) \leq 1+ {c} \cdot {\rho}\cdot r\,.
\end{equation}

 If, in addition, $L$ is a $\mathcal C^{1,1}$-submanifold then,  for all $q \in L\cap  B_{r}(p)$,
\begin{equation} \label{eq: cr}
|T_{p}L-T_{q} L|\leq   {c} \cdot  {\rho}  \cdot   r\;.
\end{equation}
\end{lem}

\begin{proof}
Upon rescaling, it suffices to prove the statement just for $\rho=1$.

The distance function  $d_L$ to $L$ is semiconvex on $B_{10} (p)$, see 
\cite[Proposition 1.1, Theorem 1.8]{KL-both}.  Moreover, as shown in  \cite{Kleinjohann}, see also the proof of   \cite[Proposition 1.1, Proposition 1.3]{KL-both},  the semiconvexity constant depends only on the curvature bound and the reach.  Thus, 
there is a universal constant 
$C$ 
such that  $d_L$ is $C$-semiconvex on $B_9(0)$. 

The gradient flow $(x,t) \to \Phi (x,t)$ of $-d_L$  retracts $B_4(p)$ along the shortest geodesics to $L$.   The $C$-semiconcavity of $-d_L$ implies that the map
 $x\to \Phi (x,t)$ is $e^{C\cdot t}$-Lipschitz continuous  \cite[Lemma 2.1.4]{Petrunin-semi}.

We find a constant $c=c(C)$ with $e^{Ct} \leq 1+c\cdot t$, for all $|t|\leq 4$.  Since $\Phi (x,r) =\Pi^L(x)$, for  $d(L,x)\leq r$,  we obtain \eqref{eq: lip},
for all $0<r\leq 4$.

We turn now to  \eqref{eq: cr}.  The existence of some constant $c=c(L,M)$ satisfying \eqref{eq: cr} is equivalent to the property that $L$ is a $\mathcal C^{1,1}$-submanifold. The claim is that $c$  can be chosen independently  of $L$ and $M$  

We fix a sufficiently small (but universally chosen) $1>>\delta >0$ to be determined later.
It is sufficient to prove \eqref{eq: cr} for all 
$r <\delta$. 

We fix  $B:=B_{\frac 1 {10}} (p)$ and some 
 distance coordinates $\Psi:B \to U\subset \R^n$. The distance coordinates are uniform $\mathcal C^{1,1}$  in the following sense, \cite[Theorem 13.2]{Ber-Nik}, \cite[Section 3]{KL-both}. For some constant $C_1=C_1 (C)$:
\begin{itemize}
\item The map $\Psi$ is $C_1$-biLipschitz.
\item For any  $v_1,v_2$ in the tangent bundle  $TB$
$$\frac 1 {C_1}\cdot |v_1-v_2| \leq|| d\Phi (v_1)- d\Phi (v_2)|| \leq C_1 
\cdot |v_1-v_2|\;.$$ 
\item The function   $d_L\circ \Psi^{-1} :U \to \mathbb R$ is $C_1$-semiconvex.
\end{itemize}

The first two properties imply that it suffices to prove the estimate \eqref{eq: cr} on
$L':=\Psi (L)\subset  U$ instead of $L\subset M$.   The third property  implies by \cite[Lemma]{Bangert}, that 
 the closest-point projection onto $L'\subset U$ is uniquely defined on $
B_{\delta '} (p')\subset U$ with $p'=\Psi (p)$, once  $\delta ' =\delta' (C_1)$ is small enough.

Then for some $\delta _1=\delta _1 (\delta')$  the intersection  $K$ of $L'$ with the  closed ball
$B':=\bar B_{\delta _1} (p')$ is a compact set of reach $\geq \delta _1$ 
in the Euclidean space $\mathbb R^n$,  \cite[Lemma 3.4]{Rataj}.
  Thus, it is a CAT$(\kappa)$ space with respect to its intrinsic metric, for some 
$\kappa=\kappa (\delta _1)$  \cite{Ly-reach}. 

 Since $L'$ is a manifold, local geodesics in $K$ starting in $p'$ are extendable as local geodesics  until the relative boundary of $K$ in $L'$ \cite[Theorem 1.5]{LSchr}. Moreover,  these local geodesics are  minimizing  in $K$ on intervals of  length $\delta =\delta (\kappa)$, due to the  CAT($\kappa)$ property. Now a uniform Lipschitz estimate for the map $x\to T_xL'$  in $L'\cap B_{\delta} (p')$ is a direct consequence of
\cite[Proposition 2.4]{L-note}.
\end{proof}

\section{Basics on submetries} 
\subsection{(Local) Submetries} \label{subsec: loc}
Recall that
 $P:X\to Y$ is a  \emph{submetry} if for any $x\in X$ and any  $r>0$ the equality
$P(B_r(x))= B_r(P(x))$ holds. 

The map $P$ is called a \emph{local submetry} if for any $x\in X$ there exists some $s>0$ such that the condition $P(B_r(z))= B_r(P(z))$ holds true for any $z\in B_s(x)$ and any $r<s$. 
 We call $X$ the \emph{total space} and $Y$ the \emph{base} of  the local submetry $P$.

 $P$ is a local submetry if and only if it is locally $1$-Lipschitz and locally
$1$-open.
A restriction of a (local) submetry $P:X\to Y$  to an open subset $O\subset X$ is a local submetry $P:O\to Y$.  A local submetry $P:X\to Y$ is a global submetry, if   $X$ and $Y$ are length spaces and $X$ is proper \cite[Corollary 2.9]{KL}.

Let $P:X\to Y$ be a local submetry and  let $X$ be a length space.
Replacing  $Y$ by $P(Y)$ 
 we may  assume that the local submetries are surjective.  
Replacing the metric on $Y$ by the induced length  metric, $P$ remains a local submetry
\cite[Corollary 2.10]{KL}.
 Thus, we may  assume without loss of generality that the base space $Y$ is a length space.

For a local submetry $P:X\to Y$, a rectifiable curve $\gamma :I\to X$  is \emph{horizontal} (with respect to $P$) if $\ell (\gamma)=\ell (P\circ \gamma)$.

\subsection{Structure of the base} \label{subsec: basestructure}
From now on let $M$ denote a  manifold with locally bounded curvature.
Let $P:M\to Y$ be a surjective local submetry.  Let $y\in Y$ be arbitrary.

  There exists some $r=r(y)>0$, such that
 any geodesic $\gamma:[0,t] \to Y$ starting in $y$ can be extended  to a geodesic $\gamma :[0,r] \to Y$ up to the distance sphere  $\partial  B_r(y)$
\cite[Theorem 1.3]{KL}.
 In this case we will say that the \emph{injectivity radius  at $y$ is at least $r$}.  Under the above assumptions, any point $y'\in B_r(y)$ is connected to 
$y$ by a unique geodesic.

Set $ m=\dim (Y)$.
Then  $Y$ admits a canonical  decomposition  $Y= \cup _{l=0} ^m Y^l$ into \emph{strata} $Y^l$.  Here, $Y^l$ is the set of all points $y\in Y$, for which  the tangent space $T_yY$ splits off $\R^l$ but not $\R^{l+1}$ as a direct factor.  
$Y^l$ is an $l$-dimensional  manifold with a canonical $\mathcal C^{1,1}$-atlas, which is  locally convex
 in $Y$, \cite[Theorem 1.6]{KL}. 
The metric on $Y^l$ is given  by a Lipschitz continuous Riemannian metric;  the tangent space $T_yY^l$ is the maximal Euclidean factor of 
$T_yY$ \cite[Theorem 11.1]{KL}.

For any point $y\in Y^l$, there exists some $r_0=r_0 (y)>0$ with the following properties  \cite[Lemma 10.1, Theorem 11.1]{KL}:
 The open ball $B_{2r_0}(y)$ does not contain  points in $\cup _{i=0}^{l-1}Y^i$
and, for  any $y'\in B_{r_0}(y) \cap Y^l$, the injectivity radius  at $y'$ is at least $r_0$.

\subsection{Fibers}  
Let $P:M\to Y$ be a surjective local submetry.  Let $y\in Y^l\subset Y$.  Then the  fiber $L=P^{-1} (y)$ and the preimage
$S=P^{-1} (Y^l)$  are subsets of positive reach in $M$ \cite[Theorems 1.1,  1.7]{KL}.

Neither $L$ nor $S$ have to be manifolds. 
However, for every $y\in Y\setminus \partial  Y$ the fiber $L=P^{-1} (y)$ is   a $\mathcal C^{1,1}$-submanifold of
$M$ \cite[Theorem 1.8]{KL}.  In particular, this applies to all
$y\in Y^m$ with $m=\dim (Y)$.

\subsection{Infinitesimal structure}
Let  $P:M\to Y$ be  a local submetry, let $x\in M$  be arbitrary,  $y=P(x)$ and denote by $L$ the fiber
$P^{-1} (y)$.

There exists a  differential $D_xP:T_xM\to T_yY$, which is itself a submetry. 
The tangent  space $T_xL$ is  the preimage $D_x P^{-1} (0)$ and it is a convex cone in  $T_xM$  \cite[Proposition 3.3, Corollary 3.4]{KL}.
 We call   $T_xL$  the \emph{vertical space} at  $x$ and denote it by $V^x$.

The \emph{horizontal space} $H^x$ is the dual cone of $T_xL$ in $T_xM$.  The cone  $H_x$ consists of all  $h\in T_xM$ such that
$||h||= |D_xP(h)|$, where  $|\cdot |$ on the right side denotes the distance  to the origin of $T_yY$.

A Lipschitz curve $\gamma:I\to M$ is horizontal if and only if the vector $\gamma'(t)$ is horizontal, for almost all $t\in I$.

\section{$P$-almost flatness} \label{sec-P}
\subsection{A single point}
Let $P:M\to Y$ be a local submetry, where $M$ has locally bounded curvature.  Fix $y\in Y$  and consider $L=P^{-1} (y)$.
Let $x\in L$ be such that a neighborhood of $x$ in $L$ is a $\mathcal C^{1,1}$-submanifold.

For any sequence $z_j \to x$ in $M$, any Gromov--Hausdorff limit of (any subsequence of) the vertical spaces $V^{z_j}$ contains $V^{x}$,  \cite[Corollary 8.4]{KL}.  Thus, we find some $r_1 >0$  such that $B_{r_1} \cap L$  is a $\mathcal C^{1,1}$-submanifold and   such that the following holds true: For any $z\in B_{r_1} (x)$ and any unit vector $v\in V^{x}$, there exists some $v'\in V^z$ with
\begin{equation} \label{eq: eps+}
|v-v'| \leq \varepsilon.
\end{equation}

We call $r_1$ as above the \emph{vertical semicontinuity radius} of $P$ at $x$.

\begin{lem} \label{lem: summ}
There exists 
 $\lambda >0$ 
with the following property.

Let $r_1>0$ be given. Let $M$ be a manifold with geometry bounded at $x$ by  $\frac 1 {r_1}$.  
Let $P:M\to Y$ be a local submetry, $y=P(x)$ and $L=P^{-1} (y)$.   Let the vertical semicontinuity radius of $P$ at $x$ and  the injectivity radius of $Y$ at $y$ be at least $r_1$.

Then, upon rescaling $M$ and $Y$ by the constant $\frac {\lambda} {r_1}$, we have
\begin{itemize}
\item $M$ is almost flat at $x$.
\item If  $r\leq 10$, the closed ball $\bar B_r(y)$ is strictly convex in $Y$. 
\item If  $r\leq 10$, the projection $\Pi^L:B_{r} (x)\to L$ has  Lipshitz constant:
\begin{equation}  \label{eq: rename}
\lip (\Pi^{L}  :B_{r} (x)\to L) \leq 1+\varepsilon \cdot r\,.
\end{equation}

\item For all $x_1,x_2$ in the $\mathcal C^{1,1}$-manifold $ B_{10} (x) \cap L$ we have
\begin{equation}  \label{eq: rename2}
|V^{x_1}-V^{x_2}|\leq \varepsilon \cdot d(x_1,x_2)\;.
\end{equation}

\item  
For   any $z\in B_{10} (x)$ and any unit  $v\in V^{x}$, there is  $v'\in V^z$ with
\begin{equation} \label{eq: eps}
|v-v'| \leq \varepsilon.
\end{equation}
\end{itemize}
\end{lem}

\begin{proof}
We may assume that the constant $c$ appearing in Lemma \ref{lem: radeschi} is at least $1$. We set $\lambda := \frac  {10 \cdot c} {\varepsilon}$ and   rescale  $M$ and $Y$ with  $\frac {\lambda }   { r_1 }$.  

Upon this  rescaling, the geometry of $M$ is bounded at $x$ by $\frac 1 {\lambda} < \varepsilon$. Hence, the rescaled $M$ is almost flat at $x$.

The injectivity radius of the rescaled $Y$ at $y$  is at least $\lambda$. Therefore, the closest point projection onto $L$ (in the rescaled $M$) is uniquely defined 
in $B_{\lambda} (x)$.
Applying  Lemma \ref{lem: radeschi} we deduce  \eqref{eq: rename} and \eqref{eq: rename2}.

The last point \eqref{eq: eps} follows from the definition of the vertical semicontinuity radius and $\lambda >10$.

Finally, the statement that for some $r_0$ and all $r<r_0$ the balls $\bar B_r(y)$ are strictly convex is exactly \cite[Theorem 9.2]{KL}.  Moreover, the proof actually shows that in the present situation one can take $r_0=10$.
\end{proof}

For a local submetry $P:M\to Y$ we say that $x$ is   a  \emph{$P$-almost flat point}
if the conclusions  of Lemma \ref{lem: summ} hold true without rescaling.  
Due to Lemma \ref{lem: summ},  for any local submetry $P:M\to Y$ and any point $x\in M$,
such that a neighborhood of $x$ in the fiber $L:=P^{-1} (P(x))$  is a manifold, 
$M$ becomes  $P$-almost flat at $x$ upon some rescaling.

\subsection{Stability along strata}
The bound $r_1$ appearing in Lemma \ref{lem: summ} can be chosen locally uniformly along strata:

\begin{lem} \label{lem: stratauni}
Let $P:M\to Y$ be a local submetry. Let $L$ be a fiber $P^{-1} (y)$ and let $x\in L$ be a point, such that a neighborhood of $x$ in $L$ is a manifold. Let $Y^l$ be the startum through $y$ and  $S=P^{-1} (Y)$.    Then upon rescaling by some  $\mu =\mu (M,Y,P,x,y) >0$  the following holds:  

Any $z\in B_{10} (x) \cap S$ is a $P$-almost flat point.
\end{lem}

\begin{proof}
The curvature and injectivity radii of $M$ are bounded in a fixed ball around $x$.
The injectivity radius of $Y$ is uniformly bounded from below in a neighborhood of $y$ in $Y^l$ \cite[Theorem 11.1]{KL}.

Applying Lemma \ref{lem: summ} it remains to obtain a uniform lower bound on the vertical semicontinuity radii in a neighborhood of $x$ in $S$.

For some choice of a neighborhood $U$ of $x$ in $M$, the restriction  
$P:U\cap S \to Y^l$  is a fiber bundle \cite[Proposition 11.3]{KL}.  Thus, $U\cap S$ and  $L':=U\cap P^{-1} (y')$ for any  $y'\in Y^l$  are topological manifolds.  Since $S$ and $L'$ are subsets of positive reach
\cite[Theorem  1.1, Theorem 1.7]{KL},   both subsets $U\cap S$ and $U\cap L'$ are $\mathcal C^{1,1}$-submanifolds of $M$.

The submanifold $S\cap U$ in its intrinsic metric is a manifold with locally bounded curvature \cite[Proposition 1.7]{KL-both}.   The restriction $P:S\cap U \to Y^l$ is a local submetry
with all fibers being regular.  Thus, this restriction is a $\mathcal C^{1,1}$-Riemannian submersion \cite[Theorem 1.2]{KL}.

In particular the distribution $z\to V^z$ is continuous on $U\cap S$.  Thus, the  semicontinuity of vertical spaces in $M$  around $x$ implies the following. For a sufficiently small
$2\delta >0$, any point $x_1 \in B_{2\delta} (x)\cap S$, any $z_1\in B_{2\delta} (x)$ and any unit vector $v\in V^{x_1}$ there exists some $v'\in V^{z_1}$ such that \eqref{eq: eps+} holds true.

Thus, $\delta$ is the required  uniform bound on the vertical semicontinuity radii in a neighborhood of $x$ in $S$.
 This finishes the proof.
\end{proof}

\section{Projection onto a fiber}

We aim to strengthen  \eqref{eq: rename}, \eqref{eq: rename2}, \eqref{eq: eps}.

\begin{lem} \label{lem: 4op}
Let $P:M\to Y$ be a local submetry, let $L$ be a fiber of $P$. Assume $x_0\in L$ is a $P$-almost flat point.   Then
 $\Pi^L:L'\cap B_2 (x_0)\to L$ is locally $2$-open, for any fiber $L'$ of $P$.
\end{lem}

\begin{proof}
Consider any $z'\in L'\cap B_ 2 (x_0)$.   Set $r= d(L,z') <2$.  Consider some $\delta <\varepsilon = 10^{-4}$, so that  $B_{10\delta r} (z')\subset B_2 (x_0)$.

 Consider an arbitrary $q'\in B_{\delta r } (z') \cap L'$.  Set $q=\Pi^L (q')$. 
Consider any $p\in L$ with $t:=d(p,q) <\delta r$.  It is sufficient to find $p'\in L'$ with
 $\Pi ^L (p')=p$ and $d(p',q')\leq 2t$.

Consider the ball $B=B_3(0)$ in the horizontal space $H^p$ and  the  subset $K =\exp _p ( B_3 (0)) \subset M$. Note  that $\Pi ^L (K)=p$. Consider the  distance function $f=d_K:B_2 (x_0) \to \R$.   We are looking for a point $p'\in L'$ with $f(p')=0$ and
 $d(p',q')\leq 2t$.

By the open map theorem \cite[Lemma 4.1]{Lyt-open}, it suffices to prove 
\begin{equation} \label{eq: firstin}
f(q') \leq \frac 5 4 \cdot t
\end{equation}
 and that the  absolut gradient of the restriction  $-f:L'\to \mathbb R$  satisfies 
\begin{equation}  \label{eq: secondin}
 |\nabla _{q''} (-f)| \geq \frac 4 5 
\end{equation}
 at every point $q''\in L' \cap B_{2t} (q')$.

Consider  $h=\exp _q^{-1} (q')\in H^q$.
By \eqref{eq: rename2},    we  find some $\tilde h \in H^p$ with 
 $$|\tilde h- h|\leq   \varepsilon\cdot t \cdot r\,.$$
Then $\exp _p (\tilde h)\in K$ and from Lemma \ref{lem: section2} we deduce
$$f(q')\leq d(q', \exp _p (\tilde h))\leq t +4 \varepsilon ^2\cdot t\cdot r +2\cdot \varepsilon \cdot t\cdot r  < \frac 5 4 t\;.$$
This proves \eqref{eq: firstin}.

In order to prove \eqref{eq: secondin}, we fix a point $q''\in L' \cap B_{2t} (q')$. 
Consider a point $\hat p\in K$ with 
$$f(q'')=d_K(q'')=d(q'', \hat p)\;.$$
Denote by $w\in T_{q''} M$  the starting direction of the geodesic $q''\hat p$.

If there exists a vertical unit vector $v\in V^{q''}= T_{q''} (L')$ which encloses an angle
less than $\arccos (\frac 4 5)$ with  $w$, then the first variation formula would imply    \eqref{eq: secondin}.  

Assume on  the contrary, that such a vertical vector $v\in V^{q''}$ does not exist.
 Then there exists a  unit horizontal vector   $u\in H^{q''}$  which encloses with $w$ an angle  at most   
$$\angle (u,w) \leq \frac \pi 2 - \arccos (\frac 4 5) < 1 \;.$$

Since $x_0$ is a $P$-almost flat point, we apply \eqref{eq: eps} and \eqref{eq: rename2} and   find a unit horizontal vector $u'\in H^p$
with $|u'-u| \leq 2\varepsilon$.      Then, the angle between the parallel translate $\hat w$ of $w$ to $\hat p$ and $\hat u$ of $u'$ to $\hat p$ is at most
$$\angle (\hat u, \hat w) \leq 1+3\varepsilon \;.$$ 
Note, that $\hat w$ is just the starting direction at $\hat p$ of the geodesic $\hat p q''$.

Consider the vector $\hat h =\exp _p^{-1} (\hat p) \in H^p$ and the curve
$\eta (t):= \exp _p (\hat h + t \cdot u')$ contained in $K$ and starting at $\hat p$.
By Corollary \ref{cor: buserkarcher},  the curve $\eta $ encloses with vector $\hat u$ 
an angle less than $\varepsilon$.

Thus, the angle between $\eta$ and $\hat w$ at $\hat p$ is less than $1+4\varepsilon <\frac \pi 2$.     Now the first formula of variation
implies that $d(q'', \eta (t)) <d(q'', \hat p)$ for all sufficiently small $t$.

This contradicts the choice of $\hat p$. The contradiction finishes the proof of \eqref{eq: secondin} and of the Lemma. 
\end{proof}

Using a combination of Lemma  \ref{lem: 4op} and Lemma \ref{lem: karcher} we now  provide:

\begin{thm} \label{thm: semicont}
 $P:M\to Y$ a local submetry.  Let $x_0\in M$ be such that a neighborhood of $x_0$ in  $L=P^{-1}(P(x_0)) $ 
is a $\mathcal C^{1,1}$-submanifold.  Then there exist  $r_0>0, C>0$  with the following properties, for any $0<r<r_0$ and any  $z\in B_r(x_0)$.
\begin{enumerate}
\item  For  any  $v\in V^{x_0}$ there exists  $v'\in V^z$ with
$|v-v'| \leq C\cdot r \cdot ||v||$. 

\item For any $h\in H^z$ there exists  $h'\in H^{x_0}$ with
$|h-h'| \leq C\cdot r \cdot ||h||$.

\item  For  $L^z= P^{-1} (P(z))$, the closest-point projection $\Pi^L:L^z\cap B_r(x_0)\to L$ is $(1+Cr)$-Lipschitz and  locally $(1+Cr)$-open. 
\end{enumerate}

The numbers $r_0,C$ depend only on a bound of the geometry of $M$ at $x_0$, a bound on the  injectivity radius of $Y$ at $P(x_0)$ and the vertical semicontinuity radius of $P$ at $x_0$.
\end{thm}

\begin{proof}
After rescaling we may assume that $x_0$ is a  $P$-almost flat point.  Due to Lemma \ref{lem: summ},  the rescaling constant depend only  on a bound of the geometry of $M$ at $x_0$, a bound on the  injectivity radius of $Y$ at $P(x_0)$ and the vertical semicontinuity radius of $P$ at $x_0$.

Thus, it  suffices to find  universal constants $C,r_0>0$ satisfying (1), (2), (3),
under the assumption that $x_0$ is a $P$-almost flat point.

We are going to prove (3) first.  By the definition of $P$-almost flat points, the projection
$\Pi^L$ is $(1+\varepsilon \cdot r)$-Lipschitz on the whole ball $B_r (x_0)$, for any $r<10$.
Thus, also the restriction of $\Pi ^L$ to $L'\cap B_r(x_0)$ has the same Lipschitz constant.
It suffices to improve the openness constant of $\Pi^L$ on $L'$ provided by Lemma \ref{lem: 4op}.

Set $r_0:=\frac 1 5$.
Let $r \leq  r_0$ and  $z\in B_r(x_0)$ be arbitrary. Set $x =\Pi^L (z)$ and let $p$ be a point  on $L$ with  $a_0:=d(x,p)<\varepsilon \cdot r$.

 We are going  to find a point $q\in L^z$ satisfying $\Pi ^L (q) =p$ and 
\begin{equation}  \label{eq: rs}
d(q,z)\leq (1+ 5\cdot r)\cdot a_0.
\end{equation}

 Extend the geodesic $xz$ to a point $\tilde z$ with $d(x,\tilde z) = 1$.  Lemma \ref{lem: 4op} provides  a point $ \tilde q$ on the fiber $\tilde L$ through $\tilde z$ such that $d(\tilde q,\tilde  z)\leq 2a_0$ and $\Pi^L (\tilde q)=p$.
The geodesics $\gamma:=x\tilde z$ and  $\eta:=p\tilde q$ are horizontal and $P\circ \gamma= P\circ \eta $.
 
Consider the point $q$ on the $\eta $ with 
$d(p,q)=d(x,z)$. Then $P(q)=P(z)$, hence $q\in L^z$.
From  Lemma \ref{lem: karcher} we deduce \eqref{eq: rs},
finishing the proof of (3).

In order to prove (1),  we  fix   $r\leq \frac 1 5 $ and    $z\in B_r (x_0)$. Set again $x=\Pi^L (z)$.
 Due to  \eqref{eq: cr},  we have $|V^x-V^{x_0}|\leq 2\varepsilon r$.

Consider an arbitrary unit vector $v\in V^{x_0}$. We find a unit 
vector $\hat v \in V^x$ with $|\hat v -v| \leq 3\varepsilon r$.

 For any sufficiently small $\delta >0$ consider a point $p_{\delta}\in L$
with $d(x, p_{\delta})=\delta$, such that the geodesic $xp_{\delta}$ starts in 
a direction $v_{\delta}\in T_{x} M$ with $| v_{\delta} -\hat v| \leq \varepsilon r$. 

Extend as above the geodesic $xz$ until a point $\tilde z$ with $d(x,\tilde z)=1$.
Due to Lemma \ref{lem: 4op} we find a point $\tilde q_{\delta}$ in the fiber $L^{\tilde z}$ of $\tilde z$ with $d(\tilde q, \tilde z)\leq 2\delta$.   Let $q_{\delta}$  be the point on 
the geodesic $p_{\delta} \tilde q_{\delta}$ with $d(p_{\delta}, q_{\delta})=d(x,z)$.

Then $q_{\delta} \in L^z$. From Lemma  \ref{lem: karcher}, we deduce that the 
starting direction $v_{\delta}$ of the geodesic $zq_{\delta}$ satisfies 
$$|v_{\delta}-\tilde v| \leq 5\cdot  r\,.$$

The directions $v_{\delta}$ subconverge to a vertical direction $v'\in V^z$ such that  $|v'-v|\leq 6r$. This proves (1).

 By duality of horizontal and vertical cones, (1)  implies (2).
\end{proof}

We now easily deduce:
\begin{proof}[Proof of Proposition \ref{prop: almsubm}]
We  cover  the compact manifold fiber $L$ by finitely many  balls  as provided by Theorem \ref{thm: semicont}. Choosing a tubular neighborhood $B_{r_0} (L)$  of $L$
contained in the union of these balls, we obtain the conclusion directly  from Theorem \ref{thm: semicont}(3).
\end{proof}

\section{Exponential map in the base}

\subsection{Exponential map in the base}
The following result is a  localization  of Theorem \ref{thm: exp}.

\begin{prop} \label{prop: bil}
There exists some $\mu >1$ with the following properties.   Let  $P:M\to Y, x\in M,  y=P(x)$
and $r_1$ be as in Lemma \ref{lem: summ}. Then for 
$r_0:=\frac {r_1} {\mu}$ and $C:=\mu \cdot  r_1^2$ 
 and all $r<r_0$, the exponential map
$\exp _y :B_r(0) \to B_r(y)$ is $(1+C r^2)$-bilipschitz on the ball $B_r(0)\subset T_yY$.
\end{prop}

\begin{proof}
The statement is invariant under rescalings. Upon a rescaling we may assume that $x$ is a $P$-almost flat point. 
  Let $C,r_0$ be as provided by Theorem \ref{thm: semicont}. Upon a further rescaling, depending only on $C$ and $r_0$, we may assume that the constants satisfy $r_0=1$ and $C=\frac 1 {10}$.
 It suffices to prove that, for any $r< \frac 1 {10}$,   the map  $\exp _y :B_r(0) \to B_r(y)$ is $(1+ 2\cdot   r^2)$-bilipschitz on the ball $B_r(0)\subset T_yY$.  We fix $r<\frac 1{10}$.

The ball $\bar B_r (y)\subset Y$  inherits the lower curvature bound  $-\frac 1 {100} \varepsilon ^2$ from the ball $\bar B_r(x) $, \cite[Proposition 3.1]{KL}. 
By Toponogov's theorem, 
the exponential map $\exp _y$ is $(1+\varepsilon ^2 r^2)$-Lipschitz  on $B_r(0)\subset T_yY$.
It remains to bound the Lipschitz constant of $\exp_ y^{-1}$ on $B_r(y)$.

Since $M$ is almost flat at $x$, 
the exponential map $\exp _x :B_r(0) \to B_r(x)$ is $(1+\varepsilon ^2 r^2)$-bilipschitz on the ball $B_r(0)\subset T_xM$, for any $r\leq 1$.

 We consider the ball $Q=B_r(0) \subset H^{x}$  in the horizontal space at $x$  and its exponential image 
$Z=\exp _x (Q)\subset M$.   For all $h\in Q\subset H^x$ we have, \cite[Proposition 7.3]{KL}:
$$\exp _y\circ D_xP (h) =P\circ \exp _x (h)\,.$$
 Thus, the map $P:Z\to Y$ can be written as
$$P=  \exp _y \circ D_xP   \circ \exp _x^{-1}\,.$$ 

The map $\exp _x^{-1}:Z\to Q$ is locally $(1+\varepsilon ^2 r^2)$-bilipschitz and 
the map $D_xP:H^x\to T_yY$ is a local submetry.  If we knew that $P:Z\to Y$ is locally 
$(1+  r^2)$-open, we would infer 
 that 
$\exp _y$ is locally $(1+2 \cdot   r^2)$-open.  
 Since $\exp _y$ is a homeomorphism and $(1+\varepsilon ^2 \cdot r^2)$-Lipschitz,  this would
prove that $\exp _y$ is locally $(1+2  \cdot r^2)$-bilipschitz.

It remains to prove that $P:Z\to Y$ is locally $(1+ r^2)$-open.

  Thus, consider any $h\in Q$ and $z=\exp _x(h)\in Z$.
 Set $t_0:= \varepsilon \cdot (r- |h|)$.
Let $z_1 = \exp _x (h_1) \in Z$ with $d(z,z_1) <t_0$ be given.  Set $y_1= P(z_1)$ and let
$y_2 \in Y$ be such that $t:= d(y_1,y_2)<t_0$.  We need to find a point $z_2\in Z\cap P^{-1} (y_2)$, such that  $d(z_1,z_2) \leq (1+ r^2) \cdot t$.

We set $L':=P^{-1} (y_2)$ and denote by $f$ the distance function $f:=d_{L'}$.
We are looking for $z_2\in Z$ with $f(z_2)=0$ and 
$d(z_1,z_2) \leq (1+  r^2) \cdot t$.

Since $P:M\to Y$ is a local submetry, we have 
$$f(z_1)=d(L', z_1)=d(y_1,y_2)=t\;.$$
Due to the open map theorem
 \cite[Lemma 4.1]{Lyt-open}, it suffices to prove that the absolut gradient of $-f$ on $Z$ at every point  $p\in Z\setminus L'$ is  at least $1-\frac 1 2  r^2$. 

 We fix a point  $p\in Z\setminus L'$ and a shortest geodesic from $p$ to $L'$. This geodesic is horizontal, since $L'$ is a fiber of $P$.   Let $u\in H^p$ be the starting direction of this geodesic. 

Due to Theorem \ref{thm: semicont} and the rescaling chosen above,  we find some unit vector $u'\in H^x$ with $|u-u'|\leq \frac 1 {5}  \cdot r$.   Consider $w:=\exp_x^{-1} (p)$ and   the  curve $\eta:[0,t_0]\to Z$ starting at $p$:
$$\eta (s):=\exp _x   (w+ s\cdot u')  \subset Z\;.$$ 
Due to  Corollary \ref{cor: buserkarcher}, any  starting direction  $\tilde u$ of $\eta $ at $p$ satisfies 
$$|\tilde u -u | \leq \frac 1  4 \cdot r\;.$$
Therefore, by the first formula of variation, we deduce that $-f$ grows at $p$ at least with velocity $$\cos \big(\frac 1 4 r \big) \geq 1-\frac 1 2 r^2\;.$$ 
This provides the right estimate for the absolute gradient of the function $-f:Z\to \R$ at $p$ and finishes the proof.
\end{proof}

\subsection{Applications}
As a consequence of Proposition \ref{prop: bil} we derive a local generalization of Corollary \ref{cor:  expstr}:

\begin{cor}
Let $P:M\to Y$ be a surjective local submetry. Then any startum   $Y^l$  of $Y$ is a manifold with locally bounded curvature.
\end{cor}

\begin{proof}
Let $y_0\in Y^l$ be arbitrary. 
 Consider a point $x_0\in L=P^{-1}(y_0)$
such that a neighborhood of $y_0$ in $L$ is a $\mathcal C^{1,1}$-submanifold. 

Due to Proposition \ref{prop: bil} and Lemma \ref{lem: stratauni} we find some $r_0>0$ and $C>0$ such that the following holds true.  For any $y\in B_{r_0} (y_0)\cap Y^l$ and any $r<r_0$, the exponential map $\exp_y: B_r(0) \to B_r(y)$ is  $(1+C\cdot r^2)$-bilipschitz from the ball in the tangent space $B_r(0)\subset T_y Y$.

Upon rescaling, we may assume that $M$ is almost flat at $x_0$, that $r_0=10$  and  $C=\varepsilon$.   Due to Lemma \ref{lem: stratauni}, \ref{lem:  summ},
the ball $\bar B_{10} (y_0) \cap Y^l$ is compact and convex in $Y^l$.  It inherits the 
lower curvature bound $-\varepsilon ^2$ from  $B:=B_{10}(x_0)$.
It remains to prove that the convex subset $Z:=\bar B_1 (y_0)\cap Y^l$ is CAT(1).

Consider 3 points $y,p,q $ in this $Z$. Choose  $\bar y, \bar p, \bar q$ in the round sphere   $\mathbb S^l$ of dimension $l$, such that $d(y,p)=d(\bar y, \bar p)$, $d(y,q)=d(\bar y, \bar q)$ and $\angle pyq =\angle \bar p \bar y \bar q$.   We need to prove 
$d(p,q) \geq d(\bar p, \bar q)$.

Identify the tangent spaces $T_yY^l$ and $T_{\bar y} \mathbb S ^l$ through an isometry 
$I$, which sends the starting directions of $yp$ and $yq$ to the starting directions of $\bar y\bar p$ and $\bar y \bar q$ respectively.  It suffices to prove that  the map
$$f:=\exp _{\bar y} \circ I\circ \exp _y ^{-1} :B_{1}(y)\to B_{1}  (\bar y)$$
on the ball $B_1(y)\subset Y^l$ is $1$-Lipschitz.

Due to Theorem \ref{prop: bil} the map $f$ is bilipschitz.  By construction, $f$ sends   geodesics starting at $y$ to geodesics starting at $\bar y$.  Hence $f$ sends spheres around $y$ onto spheres around 
$\bar y$ of the same radius.  

The  restriction of $\exp _{\bar y}$ to  the  concentric sphere
$\partial B_s (0)$ in $T_{\bar y} \mathbb S^n$    is $(1-\frac 1 {10} s^2)$-Lipschitz, if we equipp 
this sphere with its intrinsic metric.  Thus, the restriction 
$f:\partial B_s (y) \to \partial B_s (\bar y)$ is $1$-Lipschitz, if both spheres are equipped with their intrinsic metrics. 

The bilipschitz map $f$ is differentiable almost everywhere with linear differential, by Rademacher's theorem.  By above, at any point $z$ at which $f$ is differentiable, the differential $D_z f$ is $1$-Lipschitz.  We claim that this is enough to conclude that $f$ is $1$-Lipschitz.

Indeed, the ball $B_{1}(y_0)$ can be considered as a Euclidean subset $O\subset \R^l$ with a Lipschitz continuous Riemannian metric.  For any vector $v$ in $\R^l$, Fubini's theorem implies that  for \emph{almost every}  segment $\gamma$ in $O$ in direction of $v$, the length
of $\gamma$ in $Y$ is not less than the length of $f\circ \gamma$ in $\mathbb S^n$. On segments parallel 
to $v$ in $O\subset Y$, the length functional is \emph{continuous} with respect to uniform
convergence.  On the other hand,  the length of the images $f\circ \gamma$  is (as always) 
lower semi-continuous.  Thus, by a limiting procedure, the length of $f\circ \gamma$
is not larger than the length of $\gamma$ for \emph{every} segment $\gamma$ in the direction of $v$.  Therefore, the map $f$ is $1$-Lipschitz and $B_{1} (y_0)$  has curvature at most $1$.
\end{proof}

Another consequence of Theorem \ref{prop: bil} is the following:

\begin{cor} \label{lem: locinf}
Let $P:M\to Y$ be a surjective local submetry as above.
 Let $y\in Y$ be an arbitrary point, let $r$ be smaller than the injectivity radius of $y$ and let 
$v\in T_yY$ be a vector with $|v| <r$.  Then the tangent cones  $T_v (T_yY)$ and 
$T_{\exp _y(v)} Y$ are isometric.

In particular, if $\exp _y (v)$ is contained in the $l$-dimensional stratum $Y^l$ then 
$v$ is contained in the $l$-dimensional stratum $(T_yY)^l$.
\end{cor}

\begin{proof}
Consider the geodesic $\gamma _v:[0,r)\to Y$ in the direction of $v$ parametrized by arclength.    For $t\in (0,r)$, the tangent spaces at  $\gamma _v (t)$   do not depend on $t$,   \cite{Petruninpar}.  Moreover,  the tangent space $T_v (T_yY)$ in the Euclidean cone $T_yY$  is isometric  $T_{s\cdot v} (T_yY)$, for all $s>0$.

Due to Proposition \ref{prop: bil},  for small $s>0$,  a neighborhood of $(s \cdot v)$ in $T_yY$ is $(1+C s^2)$-bilipschitz to a neighborhood of $\exp _y(s\cdot v)$ in $Y$, for some $C$ independent of $s$.    Rescaling, letting $s$ go to $0$ and using that the tangent cones at $s\cdot v$ respectively at $\exp_y (s\cdot v)$ do not depend on $s$, we deduce the claim.
\end{proof}

\section{Transnormal submetries} \label{sec: trans}

\subsection{Horizontal geodesics}
Recall that a local submetry $P:M\to Y$ is  transnormal if all fibers  of $P$ are $\mathcal C^{1,1}$-submanifold of $M$.
A local submetry $P:M\to Y$ is 
 transnormal if and only if any local  geodesic  $\gamma :I\to M$ is horizontal once  $\gamma'(t)$ is horizontal for some $t\in I$,
 \cite[Proposition 12.5]{KL}.  In this case, 
for any horizontal local geodesic $\gamma :I\to M$, the projection $\bar \gamma := P\circ \gamma :I\to Y$ is 
a discrete concatenation of geodesics in $Y$, \cite[Corollary 7.2]{KL}.

The following Lemma is  stated as  \cite[Proposition 12.7]{KL}  for global submetries, but the proof remains unchanged in the local case:
\begin{lem} \label{lem: eqproj}
Let $P:M\to Y$ be a transnormal local submetry. Let $\gamma _1, \gamma _2:I\to M$ be  horizontal local geodesics.  Set $\bar \gamma _i:= P\circ \gamma _i$.
Assume that, for some $t\in I$, we have $\bar  \gamma _1(t)=\bar  \gamma _2 (t) $ and $\bar  \gamma _1 '(t)=\bar  \gamma _2 ' (t) $.  
Then $\bar \gamma _1$ and $\bar \gamma _2$ coincide on $I$.
\end{lem}

We can now provide

\begin{proof}[Proof of Corollary \ref{cor: last}]
The statement is local. We may assume that $I$ is a compact interval $[a,b]$.
The projection $\bar \gamma :[a,b]\to Y$ is a finite concatenation of geodesics $\bar \gamma :[s_i,s_{i+1}] \to Y$, for $a=s_0<...<s_k=b$.

For all $t\in (s_i,s_{i+1})$ the 
tangent spaces $T_{\bar \gamma (t)} Y$ are pairwise isometric  \cite[Theorem 1.1]{Petruninpar}. By  Corollary \ref{lem: locinf}, they are also isometric to  $T_{\bar \gamma ^+(s_i)} (T_{\bar \gamma (s_i)} Y)$ and to $T_{\bar \gamma ^- (s_{i+1})} (T_{\bar \gamma (s_{i+1})} Y)$.
 Here and below $\bar \gamma ^{\pm} (s)$ denotes the outgoing and the incoming direction of $\bar \gamma $  in $T_{\bar \gamma (s)} Y$.

On the other hand, for any $t\in (s_i,s_{i+1})$, the incoming and the outgoing directions
$\bar \gamma ^{\pm (t)}$ are  contained in the line factor of $T_{\bar \gamma (t)} Y$.
Hence,  $T_{\bar \gamma '(t)} (T_{\bar \gamma (t)}  Y)$ is isometric to $T_{\bar \gamma (t)} Y$, for any such $t$.

It only remains to prove that for any $s=s_1,...,s_{k-1}$,  the two tangent cones 
$T_{\bar \gamma ^{\pm}(s)} (T_{\bar \gamma (s)} Y)$ are isometric two each other.  In order to prove this, it suffices to find, for any such $s$, an isometry $I: T_{\bar \gamma (s)} Y\to T_{\bar \gamma (s)} Y$ which sends $\bar \gamma ^{+}$ to $\bar \gamma ^{-}$.

In order to find  such $I$, we consider $x:=\gamma (s)\in M$ and the differential $D_{x} P: T_xM \to T_yY$. 
 The restriction of $D_xP$   to the unit sphere $K$ in the horizontal space $H^x$ is a transnormal submetry $D_xP:K\to \Sigma _y Y$, onto the space of directions at $y$ \cite[Proposition 12.5]{KL}.  

The incoming and the outgoing  directions $\gamma^{\pm} (t)\in K$ satisfy 
$\gamma^+(t)=-\gamma ^-(t)$ and $D_xP (\gamma ^{\pm } (t))= \bar \gamma ^{\pm} (t)$.

Due to  \cite[Proposition 12.7]{KL},
the decomposition of $K$ into the fibers of the submetry $D_xP$ is  equivariant under the multiplication of $K$ with $-1$.
Thus, $-Id: K\to K$ induces an isometry $\overline{- Id}:\Sigma _y Y\to \Sigma _y Y$.
The cone over this isometry $\overline {-Id}$ is the required isometry $I:T_y Y\to T_y Y$, which satisfies  $I(\bar \gamma ^{+})=\bar \gamma ^{-}$.

This proves the claim and implies that  the spaces $T_{\bar \gamma '(t)} (T_{\bar \gamma (t)} ) Y$ are pairwise  isometric.

Let now  $l$ denote the dimension of the maximal Euclidean factor of the pairwise isometric spaces 
 $T_{\bar \gamma '(t)} (T_{\bar \gamma (t)}  Y)$. Then, for all $t\neq s_0,s_1,...,s_k$ as above,  the iterated tangent cone   $T_{\bar \gamma '(t)} (T_{\bar \gamma (t)}  Y)$ is isometric to $T_{\bar \gamma (t)}  Y$. By definition, $\gamma (t)$ is contained in 
$Y^l$, for all such $t$.
\end{proof}

\subsection{Holonomy map along a horizontal geodesic}
Let $P:M\to Y$ be a transnormal local submetry. Let $\gamma :[a,b] \to M$ be a horizontal local geodesic
with projection $\bar \gamma =P\circ \gamma$.
Due to Corollary \ref{cor: last}, there exist some $1\leq l\leq m$, such that 
$\bar \gamma (t)\in Y^l$, for  all but finitely many times $t\in [a,b]$.
For $t\in [a,b]$, consider the  fiber
$$L^t:= L^{\gamma (t)} =P^{-1} (P(\gamma (t)))\;.$$

Set $S=P^{-1} (Y^l)$.
Let $t\in [a,b]$ be such that   $\bar \gamma (t) \in Y^l$.  Set $x=\gamma (t) \in L^{t} \subset S$. Then
 $S$ is a $\mathcal C^{1,1}$-submanifold of $M$ and the restriction $P:S\to Y^l$ is a $\mathcal C^{1,1}$ Riemannian submersion.  Therefore, the normal vector $\gamma'(t)\in T_x S$ to $L^t$ 
 extends to
a unique  locally Lipschitz continuous normal field $z\to \nu _z \in T_zS$ along $L^{t}$, such that 
$$D_zP (\nu _z)= D_xP (\nu _x)=D_xP (\gamma'(t))\;.$$

Denote by $Q^t$ the set of all   $z\in L^{t}$ such that the geodesic
$\gamma ^{z} :[a,b]  \to M$ with $(\gamma^{z}) ' (t)=\nu _z$ is defined.
Then $Q^t$ is  an open subset of $L^t$ and, if $M$ is complete, $Q^t=L^t$.
Due to Lemma \ref{lem: eqproj}, for all $z\in Q^t$ 
$$P\circ \gamma ^{z} =P\circ \gamma =\bar \gamma\;.$$

For all $s\in [a,b]$ we   obtain a map $Hol^{\gamma}_{t,s} : Q^t\to L^s$, 
 \emph{the holonomy along $\gamma$}, 
 given as
$$Hol ^{\gamma } _{t,s} ( z):= \gamma ^z (s)\;.$$

Since $\nu$ and the exponential map on $M$ are locally Lipschitz, 
the map $Hol^{\gamma} _{t,s}$  is locally Lipschitz.  

 If $\gamma (s)\in Y^l$, then $Hol^{\gamma} _{t,s} (Q^t)=Q^s$ and 
 $Hol^{\gamma} _{t,s} $ and $Hol^{\gamma} _{s,t}$ are inverse to each other.
Thus,  $Hol^{\gamma} _{t,s} $ is locally bilipschitz in this case.

Let now $r\in [a,b]$ be arbitrary. Find some $s\in [a,b]$ such that $\bar \gamma (s) \in Y^l$ and $|s-r|$ is smaller than the injectivity radius at $\bar \gamma (r)$.
 Then 
$$Hol ^{\gamma} _{s,r} \circ Hol^{\gamma} _{t,s} = Hol^{\gamma} _{t,r}\;.$$
The  map $Hol^{\gamma}_{t,s}$ is locally bilipschitz, as we have seen above.
And the map $Hol^{\gamma} _{s,r} :Q^s\to L^r$ is the closest-point projection
to $L^r$.     Once $s$ has been chosen close enough to $r$, we can apply Theorem \ref{thm: semicont} and deduce that the map $Hol^{\gamma}_{s,r}:Q^s\to L^r$ is locally Lipschitz open. 

Alltogether we have verified the following

\begin{prop} \label{prop: hol}
In the notation above, the holonomy map  along $\gamma$
$Hol^{\gamma} _{t,r}: Q^t\to L^r$ is locally Lipschitz continuous and locally Lipschitz open.
If $\bar \gamma (r) \in Y^l$, then $Hol^{\gamma }_{t,r}$ is locally bilipschitz.
\end{prop}

\subsection{A uniform bound in terms of the volume}

We finally prove:

\begin{lem} \label{lem: uniformnon}
For any $n, k, \rho, \nu  >0 $ there exists  $r = r (n,k,\rho , \nu)>0$ with the following property.  Let $M$ be a  manifold with locally bounded curvature 
and let $P:M\to Y$ be a transnormal local submetry.  Let $n=\dim (M)$ and  $k=\dim (Y)$.
 Let the geometry of $M$ at $x$   be bounded by $\frac 1 {\rho}$ and  let the injectivity radius of $Y$ at $y=P(x)$ be at least $\rho$.  Let, finally, the volume of the ball
$ B_{\rho} (x) $ be at least   $\mathcal H^k (B_{\rho} (y))  \geq \nu\cdot \rho ^k$. Then 
the vertical semicontinuity radius of $P$ at $x$ is at least $r$.
\end{lem}

\begin{proof}[Proof of Lemma  \ref{lem: uniformnon}]
Assume the contrary and let $P_j:(M_j,x_j)\to (Y_j,y_j)$ be a contradicting sequence. Thus, the vertical semicontinuity radii $r_j$ of $P_j$ at $x_j$  converge to $0$.  Hence, there exist  $z_j\in M_j$ with $s_j=d(z_j,x_j) \to 0$ and unit vectors $v_j\in V^{x_j}$ 
such that $d(v_j,V^{z_j}) >\varepsilon$.

Hence, there exists a unit horizontal vector $h_j\in H^{z_j}$ such that $$\angle (h_j, \tilde v_j) < \frac \pi 2  -\varepsilon\,,$$
where $\tilde v_j$ is the parallel translation of $v_j$ to $z_j$.

Upon rescaling we may assume that $\rho=\varepsilon$.  In particular, $M_j$ are almost flat at $x_j$. Choosing a subsequence we may assume that $\bar B_{10} (x_j)$ converge 
to a space $M^{\infty}$ which is a closed ball of radius $10$ around the limit point $x$ of the sequence $x_j$.  Moreover, we may assume that the balls $\bar B_{10} (y_j)$ converge to an Alexandrov space $Z$ and that the restrictions of $P_j$ converge to a $1$-Lipschitz  map  $P:M^{\infty} \to Z$.   The open ball $B=B_{10} (x)$ is again a manifold with 
locally bounded curvature and the restriction of $P$ to $B$
is a local submetry onto the ball $B_{10} (y)\subset Z$.

Consider now geodesic $\gamma _j:[0,1] \to M_j$ starting in $z_j$ in the  direction of $h_j$.
Then $\gamma _j$ is a horizontal curve, by the definition of transnormality. The images $P_j (\gamma _j)$ are quasi-geodesics in $Y_j$, \cite[Proposition 3.2]{KL}.  Since the convergence $\bar B_{100} (y_j) \to Z$ is non-collapsed by the volume assumption, the limit  of the curves $P_j(\gamma _j)$ is a quasigeodesic in $Z$ \cite[Section 5.1(6)]{Petrunin-semi}.
Therefore, the limit geodesic $\gamma_{\infty}:[0,1]\to M^{\infty}$ starting in $x$ is horizontal.
Thus, its starting direction $h\in T_x M^{\infty}$ is horizontal.  

On the other hand, the fibers $L_j:=P_j^{-1} (y_j)$ converge to $L= P_j^{-1} (y)$,  \cite[Lemma 2.4]{KL}.
Since the manifolds $L_j$ are uniformly $\mathcal C^{1,1}$ in distance coordinates
by \eqref{eq: cr}, the tangent space $T_{x_j}L_j$ converge to $T_xL$.  Thus, any limit vector $v$ of $v_j$ is contained in $T_xL$. By assumptons on $h_j$ and $v_j$, the angle between the vertical vector $v$ and the horizontal vector $h$ is at most $\frac \pi 2 -\varepsilon$. This contradiction finishes the proof.
\end{proof}

As a direct consequence we obtain:

\begin{cor} \label{cor: uniformity}
For transnormal submetries $P:M\to Y$ the constants $C,r_0$ appearing in Theorem \ref{thm: exp}, Theorem \ref{thm: semicont}, Proposition \ref{prop: bil}   and Proposition \ref{prop: almsubm} depend only on a bound of the geometry of $M$ around  $L$, the injectivity radius of $Y$ at $y:=P(L)$ and
a lower volume bound of a ball in $Y$ around $y$.
\end{cor}

\bibliographystyle{alpha}
\bibliography{regular}

\end{document}